\documentclass{article}
\usepackage{amsmath,bbm} 
\usepackage{amssymb}
\usepackage[english]{babel}
\usepackage{amsthm}
\usepackage{hyperref}
\usepackage{xcolor}
\usepackage{tabularx}
\usepackage[
backend=biber,
style=alphabetic,
]{biblatex}
\usepackage{csquotes}
\newtheorem{theorem}{Theorem}[section]

\newtheorem{lemma}[theorem]{Lemma}
\newtheorem{proposition}[theorem]{Proposition}
\newtheorem*{conjecture}{Conjecture}
\newtheorem*{remark}{Remark}
\author{Tang Wuji}
\begin{document}

\title{Distribution of squarefree integers with double congruence conditions}

\maketitle
\begin{abstract}
    Given a positive integer $M$ with reduced residue system $\Phi_M$, we partition the set of squarefree integers according to some congruence conditions on the number of prime factors in each class $\phi \in \Phi_M$, and prove two relevant equidsitribution results in this setting. We discover the bias phenomenon in some special cases, and formulate a conjecture with numerical support.
\end{abstract}

\section{Introduction}
%\textcolor{blue}{change k into M and a mod k into b mod M}
Let $S(x;b,M)$ denote the number of squarefree numbers that are $\leq x$ and congruent to $b$, modulo $M$, where $(b,M)=1$. In 1885 Gegenbauer\cite{zbMATH02514898} proved that squarefree integers have a density of $\frac{6}{\pi^2}$ among all integers, that is $$S(x;0,1)\sim\frac{6x}{\pi^2}\quad (x\rightarrow\infty).$$
Landau\cite{Landau} proved in 1909 that 
\begin{equation}\label{0}
    S(x;b,M)\sim\frac{6x}{\pi^2}\cdot\frac{1}{M}\cdot\prod_{p|k}\left(1-\frac{1}{p^2}\right)^{-1}\quad (x\rightarrow\infty).
\end{equation}
In 1974 Hooley\cite{Hooley1975} refined this result to 
$$S(x;b,M)=\frac{6x}{\pi^2}\cdot\frac{1}{M}\cdot\prod_{p|M}\left(1-\frac{1}{p^2}\right)^{-1}+O\left(\sqrt{\frac{x}{M}}+M^{\frac{1}{2}+\epsilon}\right).$$
This is the best possible individual error term bound we have obtained so far. On the other hand, the square mean of the error term, i.e.,
$$M_2(x,M):=\sum_{b\in(\mathbb{Z}\slash M\mathbb{Z})^\star}\left|S(x;b,M)-\frac{6x}{\pi^2}\cdot\frac{1}{M}\cdot\prod_{p|M}\left(1-\frac{1}{p^2}\right)^{-1}\right|^2,$$
are studied by many people including Valentin Blomer, Ramon M. Nunes, Pierre Le Boudec. The best estimates can be found in \cite{nunes2014squarefreenumbersarithmeticprogressions} and \cite{boudec2014distributionsquarefreeintegersarithmetic}. 

Here we will consider a new distribution pattern of squarefree integers.
Let $M$ be an integer $\geq 3$, $m:=\phi(M)$, where $\phi$ is the Euler totient function, that is $m=|(\mathbb{Z}\slash M\mathbb{Z})^{\star}|$.
Let $\{m_1,\dots,m_m\}\subset \mathbb{Z}^m$ be a set of representatives of $(\mathbb{Z}\slash M\mathbb{Z})^{\star}$. 
Given $N$ squarefree and coprime with M, write $N=p_1\dots p_k$,where $p_1,\dots,p_k$ are the prime divisors of $N$. Define $n_i=n_i(N):=\#\{p_i\equiv m_i \pmod{M}\}$.

The first main result of this paper is an equidistribution.
\begin{theorem}\label{a}
Let $\{ l_i\}_{1\leq i \leq m}, l_i\in \mathbb{Z}_{\geq 1}$, and $a=(a_1,\dots,a_m)$ where $0\leq a_i\leq l_i-1$. Let $L:= \prod_{i=1}^{m} l_i$.
Define:
\begin{multline*}
    A_a(x):=\{N\text{ is squarefree and coprime with }M.\\
    n_i \equiv a_i \pmod{l_i},\quad N\leq x\}.
\end{multline*}
Then we have
$${|A_a(x)|}\sim\frac{6x}{\pi^2\cdot L\cdot \prod_{\substack{p\ prime\\p|M}}(1+p^{-1})},\ x\rightarrow\infty.$$
\end{theorem}

Theorem \ref{a} shows that $\{A_a(x)\}$ are equal-distributed. We also want to consider bias among $\{A_a(x)\}$, that is, given distinct $a,\ a'$, does the sign of $|A_a(x)|-|A_{a'}(x)|$ stablize as $x\rightarrow \infty$? Generally, there is no bias among $\{A_a(x)\}$. However in \S\ref{7} we discover that under certain conditions there should be a strict bias among $\{A_a(x)\}$. The computation there leads us to raise the following conjecture.
\begin{conjecture}
When $\phi(M)\geq 3, l_i\leq 2$ for all $i$, there exists a strict bias among $\{A_a(x)\}$. That is, given distinct $a,\ a'$, the sign of $|A_a(x)|-|A_{a'}(x)|$ always stablizes as $x\rightarrow \infty$.
\end{conjecture}

%However we will prove that given proper condition for $\{l_i\}$ and $M$, it can be shown that sometimes there are strict bias among $\{A_a(x)\}$, and sometimes there are not. We conjecture in section \ref{5} about the conditions for a strict bias among $\{A_a(x)\}$.

The second main result of this paper studies a "finite" version of $A_a(x)$.
%\begin{remark}
%    The problem considered here is a refinement of estimate for $S(x;a,k)$. In fact, if we choose $l_i:=|<m_i>_{(\mathbb{Z}\backslash M\mathbb{Z})^{\star}}|$ to be the order of $m_i$ in the group $(\mathbb{Z}\backslash M\mathbb{Z})^{\star}$, then elements in $A_a(x)$ are congurent to a constant modulo $M$.
%\end{remark}
%Furthermore, given distinct $a,a'$, the sign of $|A_a(x)|-|A_a'(x)|$ changes infinitely many times as $x\rightarrow \infty$.

\begin{theorem}\label{b}

Let $a=(a_1,\dots,a_m)$, where $a_i\in\mathbb{Z}_{\geq 0}$. Write $A:=\sum_{i=1}^{m}a_i$\\
Define:
\begin{multline*}
    A^a(x):=\{N\text{ is squarefree and coprime with }M.\ n_i=a_i ,\ N\leq x\}.
\end{multline*}
Then we have
$${|A^a(x)|}\sim \frac{A}{a_1!\dots a_m!\cdot m^A}\cdot x \cdot (log\ x)^{-1}\cdot (log\ log\ x)^{A-1}, x\rightarrow\infty.$$
    
\end{theorem}
\begin{remark}
    Given distinct $a=(a_1,\dots,a_m),a'=(a_1',\dots,a_m')$, under the condition $\sum a_i=\sum a_i'=A$, we have:
$$\lim_{x\rightarrow\infty}|A^a(x)|:|A^{a'}(x)|=\binom{A}{a_1,a_2,\dots,a_n}:\binom{A}{a'_1,a'_2,\dots,a'_n}.$$
Where $\binom{A}{a_1,a_2,\dots,a_n}$ is the multinomial coefficient and equal to $\frac{A!}{a_1!\dots a_m!}.$

\end{remark}
The estimate of $A_a(x)$ refines the estimate of $S(x;b,M)$. More precisely, let 
$$B(x;b,M):= \left\{N\leq x \middle|\ N\ is\ squarefree, \ N\equiv b \pmod M\right\}$$
so that $S(x;b,M)=|B(x;b,M)|$. 
In the setting of Theorem \ref{a}, we choose $l_i:=|<m_i>_{(\mathbb{Z}\slash M\mathbb{Z})^{\star}}|$ to be the  order of $m_i$ in the group $(\mathbb{Z}\slash M\mathbb{Z})^{\star}$. Let $A$ denote the set of all $a$. We can give $A$ an additive group structure by identifying $A$ with $\prod_{i=1}^{m}\left(\mathbb{Z}\slash l_i\mathbb{Z}\right)$. Then 
 $$h(a):=\prod_{i=1}^{m}m_i^{a_i} \pmod M$$ gives a group homomorphism from $A$ to $(\mathbb{Z}\slash M\mathbb{Z})^{\star}$. We claim that
 \begin{equation}
    B(x;b,M)=\bigsqcup_{h(a)=b\pmod M}A_a(x).
 \end{equation}
For $\forall N\in \bigsqcup_{h(a)=b\pmod M}A_a(x)$, $N\in A_a(x)$ for some $a=(a_1,\dots, a_m)$.
Write $N=p_1\dots p_k$, where $p_1,\dots,p_k$ are the prime divisors of $N$. Recall $n_i=\#\{p_i\equiv m_i \pmod{M}\}$, and $n_i \equiv a_i \pmod{l_i}$. Let $n_i=a_i+b_i l_i$. For $1\leq i \leq m$, let
  $$P_i:=\prod_{\substack{p|N\\p\equiv m_i \mod M}}p.$$
 Then $N=\prod_{i=1}^{m}P_i$. $P_i \equiv m_i^{n_i} \pmod M \equiv m_i^{a_i+b_i l_i} \pmod M$. Since $l_i$ is the order of $m_i$ in the group $(\mathbb{Z}\slash M\mathbb{Z})^{\star}$, $m_i^{l_i}\equiv 1 \pmod M.$ This gives $N\equiv \prod_{i=1}^{m}m_i^{a_i}\pmod M\equiv b\pmod M$, and $N\in B(x;b,M)$.

 Conversely, given $N\in B(x;b,M)$, decompose $N$ and introduce $n_i$, $a_i$, $P_i$ as above. Then $N\in A_a(x)$, with $a=(a_1,\dots,a_m)$, and
 $$h(a)= \prod_{i=1}^{m}m_i^{a_i}\equiv \prod_{i=1}^{m}m_i^{n_i}\equiv \prod_{i=1}^{m}P_i= N\equiv b\pmod{M}$$
 
 %then each $a=(a_1,a_2,\dots,a_n)$ corrosponds to solutions of $m_0=\prod_{i=1}^{m}m_i^{a_i}$.

 Assuming Theorem \ref{a}, we can deduce the main term of $S(x;b,M)$ given by \eqref{0} as follows. Define 
 $$A(m_i):=\{a=(a_1,\dots,a_m)|\  0\leq a_i\leq l_i-1,\ \forall N\in A_a,\ N\equiv m_i \pmod M\}$$
 Each $A(m_i)=h^{-1}(m_i)$ is a coset of $A\slash A(1)$, hence independently of $i$, we have
 $$|A(m_i)|=\frac{|A|}{m}=\frac{L}{\phi(M)}=\frac{L}{M\prod_{p|M}(1-p^{-1})}.$$ 
 This gives $$S(x;b,M)\sim |A(b)|\cdot \frac{6x}{\pi^2\cdot L\cdot \prod_{p|M}(1+p^{-1})}=\frac{6x}{\pi^2}\cdot\frac{1}{M}\cdot\prod_{p|M}\left(1-\frac{1}{p^2}\right)^{-1},\ x\rightarrow\infty.$$

\begin{remark}
    Theorem \ref{a} and Theorem \ref{b} can be extended to $n$-free numbers.
\end{remark}

%We will show that $|A(m_i)|$ is constant for all $i$. Given a specific $i$, let $\hat{a_i}:=(0,\ldots,0,1,0,\dots,0)$, where $1$ is in the $i$-th component. Obviously $\hat{a_i}\in A(m_i)$. For any $a_i\in A(m_i)$, $a_i-$
 %A_a(x) are congurent to a constant modulo $M$.
\section{Tauberian Theorems}

We need the following lemmas which connect the pole of a Dirichlet series with the order of growth of its partial sum of its coefficients.
The first lemma is due to Selberg-Delange \cite{tenenbaum2015introduction}:\\
Let $z\in \mathbb{C}, 0<\delta\leq 1, M>0.$ We say that a Dirichlet series $F(s)$ has property $P(z;c_0,\delta,M)$ if the Dirichlet series $G(s;z):=F(s)\zeta(s)^{-z}$ may be analytically continued to $\sigma\geq 1-c_0/(1+log^+(\tau))$, and, in this domain, satifies the bound
$$|G(s;z)|\leq M(1+|\tau|)^{1-\sigma}.$$
We will say $F(s)$ has a pole of order $z$ at $s=1$. Here $s=\sigma+i \tau\in \mathbb{C}, log^+(x):=Max(0, log(x)).$\\
If $F:=\sum_{n=0}^{\infty}a_n n^{-s}$ has property $P(z;c_0,\delta,M)$ and there exist real numbers $\{b_n\}_{n=1}^\infty$ such that $|a_n|\leq b_n$ and $\sum_{n=1}^\infty b_n n^{-s}$ satisfies $P(w,c_0,\delta,M)$ for some $w\in \mathbb{C}.$ We say $F(s)$ has property $T(z,w;c_0,\delta,M).$
\begin{lemma}
    \label{1}
Given $F(s)=\sum_{n=1}^\infty a_n n^{-s}$ with property $T(z,w;c_0,\delta,M)$, we have:
\begin{equation*}
    \sum_{n\leq x}a_n = x(log\ x)^{z-1}\left[\sum_{k=0}^{N}\frac{\lambda_k(z)}{(log\ x)^k}+O(M R_N(x))\right],
\end{equation*}
    where $$R_N(x):=e^{-c_1 \sqrt{log x}}+(\tfrac{c_2 N +1}{log x})^{N+1},$$
    and $$\lambda_k(z):=\frac{1}{\Gamma(z-k)}\sum_{h+j=k} \frac{1}{h!j!}G^{(h)}(1;z)\gamma_j(z).$$

Here $\gamma_j(z)$ is defined by the Taylor series coefficient: 
\begin{equation*}
    s^{-1}((s-1)\zeta(s))^z=\sum_{j=0}^{\infty} \frac{1}{j!}\gamma_j(z)(s-1)^j.
\end{equation*}
And in the domain where $G(s;z)$ is holomorphic, we set
$$G^{(k)}(s;z):=\frac{\partial^k}{\partial s^k}G(s;z).$$
\end{lemma}

The next Lemma is due to Delange\cite{Delange1954}:
\begin{lemma}
    \label{2}
    Let $F(s):=\sum_{n=1}^{\infty}a_nn^{-s}$be a Dirichlet series with $a_n\geq0$, and converges for $\sigma>1$. Suppose $F(s)$ is holomorphic at all points of the line $\sigma=1$ except at $s=1$, and, in a neighborhood of $s=1$ in the half plane $\sigma>1$, we have:
    \begin{equation*}
        F(s)=(s-1)^{-w-1}\sum_{j=0}^{q}g_i(s)\left(log(\tfrac{1}{s-1})\right)^j+g(s),
    \end{equation*}
    where $w$ is a negative integer,and $g_i(s),g(s)$ are holomorphic at $s=1$, with $g_q(1)\neq0$.
    Then as $x\rightarrow\infty$
    \begin{equation*}
        A(x):=\sum_{n\leq x}a_n\sim (-1)^{-w-1}(-w-1)!\ q\ g_q(1) x(log\ x)^w(log\ log\ x)^{q-1}.
    \end{equation*}
\end{lemma}
\section{Proof}
\subsection{Congruence restriction on the number of prime divisors}
\begin{proof}[Proof of Theorem \ref{a}]
For each $l_i$, let $e_i:=e^{\frac{2\pi i}{l_i}}$ denote a $l_i$-th root of unity. Define $J:=\{(j_1,\dots,j_m)|\ 0\leq j_i\leq l_i-1, j_i\in\mathbb{Z}\}.$

For each $j=(j_1,\dots,j_m)\in J$, define 
\begin{equation}\label{3}
    L_j(s):=\prod_{1\leq i\leq m}\prod_{\substack{p\  prime\\p\equiv m_i\text{ mod }M}}(1+e_i^{j_i}p^{-s}).
\end{equation}
Define $$L_0(s):=L_{(0,\dots,0)}(s)=\frac{\zeta(s)}{\zeta(2s)\prod_{\substack{p \  prime\\p|M}}(1+p^{-s})},$$
$$A_a:=\bigcup_{x=1}^{\infty}A_a(x),$$
$$L_{A_a}(s):=\sum_{n\in A_a}n^{-s}.$$
$L_{A_a}(s)$ can be written as a linear combination of $L_j(s)$:
$$L_{A_a}(s)=\frac{1}{L}\sum_{j\in J}\left(\prod_{1\leq i\leq m}e_i^{-a_i j_i}L_j(s)\right).$$
Especially, the coefficient of $L_0(s)$ is $\frac{1}{L}$ for all $a$.

%Define
%\begin{align*}
%        L_{A^a}^\star(s):=&\prod_{1\leq i\leq m}\prod_{\substack{p\  prime\\ p\equiv m_i\ \text{mod } M}}(1+p^{-s})^{e_i^{j_i}}\\
%=&\prod_{1\leq i\leq m}\left(\prod_{\substack{p\  prime\\ p\equiv m_i\ \text{mod } M}}\frac{1}{1-p^{-s}}(1-p^{-2s})\right)^{e_i^{j_i}}
%\end{align*}

%Since $\prod_{\substack{p\  prime\\ p\equiv m_i\ \text{mod } M}}(1-p^{-2s})$ converges for all $i$, and $\prod_{\substack{p\  prime\\ p\equiv m_i\ \text{mod } M}}\frac{1}{1-p^{-s}}$ has pole of order $\frac{1}{m}$, 
%$L_{A^a}(s)$ has pole of order $\frac{1}{m}\sum_{i=1}^{m}e_i^{j_i}$ at $s=1$.
Define$$L_j^\star(s):=\prod_{1\leq i\leq m}\prod_{\substack{p\  prime\\ p\equiv m_i\bmod M}}(1+p^{-s})^{e_i^{j_i}}.$$
Note that for a Dirichlet character $\chi$ mod $M$:
$$\log L(s,\chi) = \sum_{p,k} \frac{\chi(p^k)}{kp^{ks}} = 
\sum_p \frac{\chi(p)}{p^s} + \sum_{p,k\geq 2} \frac{\chi(p^k)}{kp^{ks}}.$$
The second term converges absolutely for $Re(s)>\frac{1}{2}$, hence in a neighborhood of $s=1$ we can write
$$\log L(s,\chi) = \sum_p \frac{\chi(p)}{p^s} + O(1).$$
Therefore we have
\begin{align*}
    \sum_{p \equiv m_i \bmod M} \frac{1}{p^s} =& 
\frac{1}{m} \sum_{p}\sum_{\chi} \frac{\chi(p)\overline{\chi}(m_i)}
{p^s}\\
= &
\frac{1}{m} \sum_{\chi}\overline{\chi}(m_i)\left(\sum_{p} \frac{\chi(p)}
{p^s}\right)
\\=&\frac{1}{m}\sum_{\chi \bmod M}  \overline{\chi}(m_i)\log L(s,\chi) + O(1).
\end{align*}

Where $\chi$ ranges over Dirichlet characters mod $M$.
For all nontrivial Dirichlet characters $\chi$, $L(s,\chi)$ converges at $s=1$ and $L(1,\chi)\neq 1$. For the trivial character $\chi_0$, $L(s,\chi_0)$ has a simple pole at $s=1$, that is to say for $Re(s)>1$:
$$log(L(s,\chi_0))=log\left(\frac{1}{s-1}\right)+O(1).$$
Hence 
$$log\left(L_j^\star(s)\right)=\frac{1}{m}\sum_{i=1}^{m}e_i^{j_i}log\left(\frac{1}{s-1}\right)+O(1).$$
Therefore, $L_j^\star(s)$ has a pole of order $\frac{1}{m}\sum_{i=1}^{m}e_i^{j_i}$ at $s=1$.

Notice that $\prod_{\substack{p\  prime\\ p\equiv m_i\bmod M}}(1+p^{-s})^{e_i^{j_i}}=\prod_{p}(1+e_i^{j_i}p^{-s}+O(p^{-2s}))$, where $O(p^{-2s})$ is bounded independently with $i$, therefore, $L_j(s)/L_j^\star(s)=\prod_{p}(1+O(p^{-2s}))$ which converges at $s=1$, hence $L_j$ has a pole of order $\frac{1}{m}\sum_{i=1}^{m}e_i^{j_i}$ at $s=1$ as well. In addition, the real part of the order of the pole is always $<1$ unless $j=(0,\dots,0)$.

Let $L_j(s)=\sum_{n=1}^{\infty} a_{j,n}n^{-s}$ be the Dirichlet series expansion, and let $\tilde{A}_{j}(x):=\sum_{n\leq x}a_{j,n}$. Then 
$$|A_a(x)|=\frac{1}{L}\sum_{j\in J}\prod_{1\leq i\leq m}e_i^{-a_i j_i}\tilde{A}_{j}(x).$$
We can now apply Lemma \ref{1}, to $L_j(s)$, and we have 
$$\tilde{A}_{j}(x)=O\left(x(log x)^{\frac{1}{m}\sum_{i=1}^{m}e_i^{j_i}-1}\right),$$
which $=o(x)$ when $j\neq (0,\dots,0)$.

$L_0(s)$ has a regular pole at $s=1$ with residue $\frac{1}{\zeta(2)\prod_{\substack{p \  prime\\p|M}}(1+p^{-1})}$ at $s=1$, 
hence we have:
$$|A_0(x)|=\frac{1}{\zeta(2)\cdot \prod_{\substack{p\ prime\\p|M}}(1+p^{-1})}\cdot x +o(x).$$
This gives 
$$|A_a(x)|=\frac{6}{\pi^2\cdot L\cdot \prod_{\substack{p\ prime\\p|M}}(1+p^{-1})}\cdot x +o(x).$$
\begin{remark}
    If we apply Lemma \ref{1} to all of $L_j(s)$ in $L_{A_a}(s)$, we can refine the error term to $O\left(x\cdot (log\ x)^{cos\left(\frac{2\pi}{max(l_i)_i}\right)-1}\right)$. However the power of $x$ in the error term can never be reduced below $1$ with this method alone. 
    Another result is that if we consider $|A_a(x)|-|A_a'(x)|$, the $L_0$ term cancels out and the resulting difference has to be of the form $x(log x)^{\frac{1}{m}\sum_{i=1}^{m}e_i^{j_i}-1}(C+O(\frac{1}{log (s-1)}))$. If given specific $a, a'$, we may directly calculate $C$, and $C\neq 0$ would imply the sign of $|A_a(x)|-|A_a'(x)|$ changes infinitely many times as $x\rightarrow \infty$. We conjecture that this is true for any distinct $a, a'$.
\end{remark}

\end{proof}
\subsection{Restriction of the number of prime divisors}

\begin{proof}[Proof of Theorem \ref{b}]
    Let $\hat{a}_i:=(0,\ldots,0,a_i,0,\dots,0)$, where $a_i$ is in the $i$-th component. 
    
    Define:
    \begin{align*}
        L_{A^a}(s):=&\sum_{n\in A^a}n^{-s},\\
        L_{A^{\hat{a}_i}}(s):=&\sum_{n\in A^{\hat{a}_i}}n^{-s}.
    \end{align*}
    Then we have: 
    $$L_{A^a}(s)=\prod_{i=1}^{m}\left(\sum_{n\in A^{\hat{a}_i}}n^{-s}\right)=\prod_{i=1}^{m}L_{A^{\hat{a}_i}}(s).$$
    Assume $a_i\geq 2$. Consider the Dirichlet series expansion for $\left(\sum_{p \equiv m_i \bmod M}\frac{1}{p^{s}}\right)^{a_i}$. $\forall n\in A^{\hat{a}_i}$, $n^{-s}$ has coefficient $a_i!$.
    On the other hand, the terms not in $ A^{\hat{a}_i}$ all contain $p^{-2s}$ for some $p \equiv m_i \bmod M$. This gives
    $$\left|\frac{1}{a_i!}\left(\sum_{p}\frac{1}{p^{s}}\right)^{a_i}-L_{A^{\hat{a}_i}}(s)\right|=O\left(\left(\sum_{p}\frac{1}{p^{2s}}\right)\left(\sum_{p}\frac{1}{p^{s}}\right)^{a_i-2}\right).$$
    Here all summations are over $p$ prime and $p \equiv m_i \bmod M$. We have
    $$L_{A^{\hat{a}_i}}(s)=\frac{1}{M^{a_i}a_i!}\left(log\ \frac{1}{s-1}\right)^{a_i} +O\left(\left(log\ \frac{1}{s-1}\right)^{a_i-2}\right).$$
    If $a_i=1$, then we have $$L_{A^{\hat{a}_i}}(s)=\sum_{p \equiv m_i \bmod M}\frac{1}{p^{s}}=\frac{1}{M}log\frac{1}{s-1}+O(1).$$
    In any case, we have$$L_{A^{\hat{a}_i}}(s)=\frac{1}{M^{a_i}a_i!}\left(log\ \frac{1}{s-1}\right)^{a_i} +O\left(\left(log\ \frac{1}{s-1}\right)^{a_i-1}\right).$$
    Hence $$L_{A^a}(s)=\prod_{i=1}^{m}L_{A^{\hat{a}_i}}(s)=\frac{1}{m^A\cdot\prod_{i=1}^{m}a_i!}\left(log\ \frac{1}{s-1}\right)^{A} +O\left(\left(log\ \frac{1}{s-1}\right)^{A-1}\right).$$
%      =&\prod_{i=1}^{m}\left(\frac{1}{a_i!}\left(\sum_{p \equiv m_i \bmod M} \frac{1}{p^s}\right)^{a_i}+O\left(\left(\sum_{n\in A^{\hat{a}_i}}n^{-s}\right)^{i-1}\right)\right)\\
%        =&\frac{1}{m^A\cdot\prod_{i=1}^{m}a_i!}(log\ \frac{1}{s-1})^{A} +O((log\ \frac{1}{s-1})^{A-1})
Apply Lemma \ref{2} to $L_{A^a}(a)$ and we have:
$${|A^a(x)|}\sim \frac{A}{a_1!\dots a_m!\cdot m^A}\cdot x \cdot (log\ x)^{-1}\cdot (log\ log\ x)^{A-1},\ x\rightarrow\infty.$$
\end{proof}

%\section{Generalization for $n$-free integers}
%A natural generalization of the result of \ref{a} is considering n-free integers. Here $n$ is a integer$\geq 2$, and $n_i$ will be the number of divisors that $\equiv m_i \bmod{M}$ counting with multiplicity.
%\begin{theorem}
    %Let $\{ l_i\}_{0\leq i \leq m}, l_i\in \mathbb{Z}_{\geq 1}$, and $a=(a_1,\dots,a_m)$ where $0\leq a_i\leq l_i-1$. Let $L:= \prod_{i=1}^{m} l_i$.
%Define:
%\begin{multline*}
    %A_a(x):=\{N\text{ is n free and coprime with }M.\\
    %n_i \equiv a_i \pmod{l_i},\quad N\leq x\}
%\end{multline*}
%Then:
%$${|A_a(x)|}\sim\frac{x}{\zeta(n) L}\prod_{\substack{p\ prime\\p|M}}\frac{1-p^{-1}}{1-p^{-ns}}, x\rightarrow\infty$$
%\end{theorem}
\section{Some possible Bias Among $A_a(X)$}\label{6}
Given $M$ and $l_i$, we consider the sign of $|A_a(x)|-|A_{a'}(x)|$, where $a$ and $a'$ are distinct. We'll show that if some $l_i\geq 3$, this sign changes infinitely many times, and if $l_i\leq 2$ for all $i$, we may expect the sign to become stable as $x\rightarrow \infty$. 
We shall justify this phenomenon for certain specific $M, l_i$.

\subsection{$\phi(M)\geq 3, l_i\leq 2$ for All $i$}\label{7}
%First we assume $m=\phi(M)\geq 3$.

After re-indexing, we may assume $l_1=\dots=l_h=2$, and $l_{h+1}=\dots=l_m=1$. If $h=0$, there would be only one element $a$ in $A$ and it makes no sense to talk about bias. So we may assume $h\geq 1$. 
We'll use $(\mathbb{Z}\slash2\mathbb{Z})^h$ to denote the elements of $A$, and let $A_1$ denote the subset of A with exactly one index being $1$, and the rest being $0$. 
For $L_j(s)$ defined in \eqref{3}, let $|j|$ denote the number of indices that are equal to $1$, then $L_j(s)$ has a pole of order $1-\frac{2|j|}{m}$ at $s=1$. Hence, among the collection of $L_j(s)$, $\{L_j(s)\}_{j\in A_1}$ has the second highest order pole of order $1-\frac{2}{m}>0$ after $L_0(s)$, which has a pole of order 1 at $s=1$.
%Especially, since $m\geq 3$, $1-\frac{2}{m}>0$. 

Given $1\leq c \leq h$, let $j_c\in A_1$ be the element whose $c$-th component equals $1$. Then 
$$L_{j_c}(s):=\left(\prod_{\substack{1\leq i\leq m\\i\neq c}}\prod_{\substack{p\  prime\\p\equiv m_i\text{ mod }M}}(1+p^{-s})\right)\prod_{\substack{p\  prime\\p\equiv m_c\text{ mod }M}}(1-p^{-s})$$
Let $\sum_{n=1}^{\infty}a_{j,n} n^{-s}$ be the Dirichlet series expansion of $L_{j_c}(s)$ and let $A_{j_c}(x):=\sum_{n\leq x} a_{j,n}$. By Lemma \ref{1} we have:
$$A_{j_c}(x)\sim c_c\cdot x(log\ x)^{-\frac{2}{m}},\ c_c=\frac{1}{\Gamma(1-\frac{2}{m})}\lim_{s\rightarrow 1}L_{j_c}(s)\zeta(s)^{-1+\frac{2}{m}}\neq 0$$
%where 

\begin{align*}
    c_c&=\frac{1}{\Gamma(1-\frac{2}{m})}\lim_{s\rightarrow 1}L_{j_c}(s)\zeta(s)^{-1+\frac{2}{m}}\\
    &=\frac{1}{\Gamma(1-\frac{2}{m})}\prod_{p|M}(1-p^{-1})^{1-\frac{2}{m}}\\
    &\lim_{s\rightarrow 1}\prod_{i\neq c}\prod_{p\equiv m_i(M)}\left((1+p^{-s})(1-p^{-s})^{1-\frac{2}{m}}\right)\cdot\prod_{p\equiv m_c(M)}(1-p^{-s})^{2-\frac{2}{m}}.
\end{align*}
We will need the following lemma for the purpose of numerically computing $c_c$.
\begin{lemma} \label{5}
We have the equality of the two following limits:
    \begin{align*}
         &\lim_{s\rightarrow 1}\prod_{i\neq c}\prod_{p\equiv m_i(M)}\left((1+p^{-s})(1-p^{-s})^{1-\frac{2}{m}}\right)\cdot\prod_{p\equiv m_c(M)}(1-p^{-s})^{2-\frac{2}{m}}\\
        =&\lim_{N\rightarrow \infty}\prod_{i\neq c}\prod_{\substack{p\equiv m_i(M)\\p\leq N}}\left((1+p^{-1})(1-p^{-1})^{1-\frac{2}{m}}\right)\cdot\prod_{\substack{p\equiv m_c(M)\\p\leq N}}(1-p^{-1})^{2-\frac{2}{m}}.
    \end{align*}
\end{lemma}
\begin{proof}
    Define 
    $$F(s):=\prod_{i\neq c}\prod_{p\equiv m_i(M)}\left((1+p^{-s})(1-p^{-s})^{1-\frac{2}{m}}\right)\cdot\prod_{p\equiv m_c(M)}(1-p^{-s})^{2-\frac{2}{m}}.$$
    Then we have
    $$log\ F(s)=\frac{2}{m}\sum_{\substack{i\neq c\\p\equiv m_i (M)}}p^{-s}+\left(\frac{2}{m}-2\right)\sum_{p\equiv m_c(M)}p^{-s}+\sum O(p^{-2s}).$$
Here $\sum O(p^{-2s})$ converges absolutely on $Re(s)>\frac{1}{2}$, hence we can exchange $\lim_{N\rightarrow \infty}\sum_{n\leq N}$ with $\lim_{s\rightarrow 1}$.

Let $a_n=\begin{cases}
    \frac{2}{m} & \text{if } n\ \text{prime and}\ n\equiv m_i (M) \text{ for some } i\neq c\\
    \frac{2}{m}-2   & \text{if } n\ \text{prime and}\ n\equiv m_c (M)\\
    0 & \text{otherwise}
\end{cases}$

Let $$L(s):=\sum_{n=1}^{\infty}\frac{a_n}{n^s}=\frac{2}{m}\sum_{\substack{i\neq c\\p\equiv m_i (M)}}p^{-s}+\left(\frac{2}{m}-2\right)\sum_{p\equiv m_c(M)}p^{-s}.$$

We need to prove: $L(s)$ can be extended to $s=1$, and 
$$\lim_{s\rightarrow 1^+}L(s)=\lim_{N\rightarrow \infty}\left(\frac{2}{m}\sum_{\substack{p\equiv m_i (M),\ i\neq c\\p\leq N}}p^{-1}+\left(\frac{2}{m}-2\right)\sum_{\substack{p\equiv m_c (M)\\p\leq N}}p^{-1}\right).$$

Define 
\begin{align*}
    s_k&:=\sum_{n=1}^{k}a_n\\
    &=\frac{2}{m}\sum_{i\neq c}\#\left\{p\leq k |\ p\ prime,\ p \equiv m_i\ (\bmod{M})\right\}+\\
    &\left(\frac{2}{m}-2\right)\#\left\{p\leq k |\ p\ prime,\ p \equiv m_c\ (\bmod{M})\right\}
%    &=\frac{2}{m}\#\{primes\equiv m_i(M)\leq k, \ i\neq c\}-\left(\frac{2}{m}-2\right)\#\{primes\equiv m_c(M)\leq k\}
\end{align*}
By the Siegel-Walfisz theorem\cite{Walfisz1936}, we have
$$\#\left\{p\leq k |\ p\ prime,\ p \equiv m_i\ (\bmod{M})\right\}=\frac{1}{\phi(M)}Li(k)+O(k\cdot \left(log\ k\right)^{-2}).$$ 
This gives $s_k=O(k\cdot \left(log\ k\right)^{-2})$.
%Assuming Riemman Hypothesis, we have $\#\{primes\ \equiv m_i(M)\ \leq k\}=\frac{1}{\phi(M)}Li(k)+O(k^{{\frac{1}{2}}+\epsilon})$ for any $\epsilon>0$. This gives $s_k=O(k^{\frac{1}{2}+\epsilon})$.

\begin{align*}
    \sum_{k=1}^{N}\frac{a_k}{k^s}&=\sum_{k=1}^{N}\frac{s_k-s_{k+1}}{k^s}\\
    &=\sum_{k=1}^{N-1}s_k\left\{\frac{1}{k^s}-\frac{1}{(k+1)^s}\right\}+\frac{s_N}{N^s}\\
    &=\sum_{k=1}^{N-1}f_k(s)+\frac{s_N}{N^s}.
\end{align*}
Where $f_k(s):=s_k\left\{\frac{1}{k^s}-\frac{1}{(k+1)^s}\right\}$, and by the mean value theorem we have:
$$\left|f_k(s)\right|\leq c\cdot s\cdot k^{-s}\cdot\left(log\ k\right)^{-2},$$
%$$\left|f_k(s)\right|\leq c\cdot s\cdot k^{-s-\frac{1}{2}+\epsilon}$$
where $c$ is a positive constant.
%This means when $Re(s)>\frac{1}{2}+\epsilon$, $\sum f_k(s)$ converges absolutely and uniformally. 
This means $\sum f_k(s)$ converges uniformally for $s\in[1,2]$. Hence we can exchange the limit and valuation at $s=1$.
\end{proof}
In particular, we have $c_c>0$, since both $log(L_{j_c}(s))$ and $log(\zeta(s))$ are real when $s>1$.

Given $a=(a_1,\dots,a_h)\in A$, we have 
\begin{equation}\label{4}
    L_{A_a}(s)=\frac{1}{2^h}\left(L_0(s)+\sum_{i=1}^h (-1)^{a_i}L_{j_i}(s)+o((s-1)^{-1+\frac{2}{m}})\right).
\end{equation}

Hence for distinct $a, a'$, $L_{A_a}(s)-L_{A_a'}(s)$'s $L_0(s)$ term cancels out, but $L_j(s)$, $j\in A_1$ terms need not. If they don't, the sign of the leading term in the asymptotic expansion of $|A_a(x)|-|A_{a'}(x)|$ is fixed. 

Note that, for $a=(0,\dots, 0)$, the coefficient of $L_{j_i}(s)$ in \eqref{4} is maximal for every $i$, while for $a=(1,\dots,1)$,  this coefficient is minimal. We deduce the following result.
\begin{proposition}
     For all $a\neq(0,\dots,0)$, $|A_{(0,\dots,0)}(x)|-|A_{a}(x)|$ is positive when $x$ is sufficiently large. For all $a\neq(1,\dots,1)$, $|A_{(1,\dots,1)}(x)|-|A_{a}(x)|$ is negative when $x$ is sufficiently large.
\end{proposition}

For given $M, l_i$, we can always justify whether or not there is a bias by numeric method. We describe the method with the following example:

Let $M=12$. Then $m=4$. Let $m_1=1,\ m_2=5,\ m_3=7,\ m_4=11$. Let $l_1=...=l_4=2$. Let $a=(1,0,0,0)$ and $a'=(0,1,0,0)$. Then $L_{A_a}(s)-L_{A_a'}(s)=\frac{1}{16}(-L_{(1,0,0,0)}(s)+L_{(0,1,0,0)}(s)+o(x^{\frac{1}{2}}))$. 
%\textcolor{blue}{A_(1,\dots,0)(x)}
We have:$$A_{j_i}(x)\sim c_i\cdot x(log\ x)^{-\frac{1}{2}},\ 1\leq i\leq 4$$
Where:
$$c_1=\frac{1}{\Gamma(\frac{1}{2})}\prod_{p|12}(1-p^{-1})^{\frac{1}{2}}\lim_{N\rightarrow \infty}\prod_{\substack{p\not\equiv 1(12)\\(p,12)=1\\p\leq N}}\left((1+p^{-1})(1-p^{-1})^{\frac{1}{2}}\right)\cdot\prod_{\substack{p\equiv 1(12)\\p\leq N}}(1-p^{-1})^{\frac{3}{2}}$$
$$c_2=\frac{1}{\Gamma(\frac{1}{2})}\prod_{p|12}(1-p^{-1})^{\frac{1}{2}}\lim_{N\rightarrow \infty}\prod_{\substack{p\not\equiv 5(12)\\(p,12)=1\\p\leq N}}\left((1+p^{-1})(1-p^{-1})^{\frac{1}{2}}\right)\cdot\prod_{\substack{p\equiv 5(12)\\p\leq N}}(1-p^{-1})^{\frac{3}{2}}$$
%$$\lim_{s\rightarrow 1}\sqrt{\frac{\prod_{\substack{p\not\equiv 1(12)\\(p,12)=1}}(1+p^{-s})\prod_{p\equiv 1(12)}(1-p^{-s})^{2}}{\prod_{p\equiv 1(12)}(1+p^{-s})}}\thickapprox 1.23232$$
Denote $c=\frac{1}{\Gamma(\frac{1}{2})}\prod_{p|12}(1-p^{-1})^{\frac{1}{2}}$, and $\bar{c}_i:=\frac{c_i}{c}$, the result for the computation is as follows (where we choose $N=p_m$, and $p_m$ is the $m$-th prime number):\\
\begin{center}
\begin{tabular}{ |c|c|c|c|c|c| } 
 \hline
$m$ & 100 & 1,000 & 10,000 & 100,000 & 1,000,000 \\ 
 $\bar{c}_1$ & 1.16546 & 1.17505 & 1.1763 & 1.1766 & 1.17678 \\ 
 $\bar{c}_2$ & 0.744375 & 0.740081 & 0.740074 & 0.740005 & 0.739976 \\ 
 \hline
\end{tabular}
\end{center}

%Its convergence is guaranteed by Merten's theorem for arithmetic progressions. \cite{keliher2024constantsmertenstheoremsprimes}

This shows that $|A_a(x)|-|A_{a'}(x)|$ is always negative when $x$ is sufficiently large.
We will write $a<a'$ in this case. 

By computing to the one million-th prime, we have calculated that $\bar{c}_3=0.854448,\ \bar{c}_4= 1.01913$, and we have the following result:
\begin{align*}
    &(1,1,1,1)<(1,0,1,1)<(1,1,0,1)<(1,1,1,0)<\\
    &(0,1,1,1)<(1,0,0,1)<(1,0,1,0)<(1,1,0,0)<\\
    &(0,0,1,1)<(0,1,0,1)<(0,1,1,0)<(1,0,0,0)<\\
    &(0,0,0,1)<(0,0,1,0)<(0,1,0,0)<(0,0,0,0).
\end{align*}

We have also computated the case for $M=5,\ 8,\ 10$, and $l_i=2$ for all $i$. Here $\bar{c}_i$ is defined similarly in the above case, where
$$A_{j_i}(x)\sim \frac{1}{\Gamma(\frac{1}{2})}\prod_{p|M}(1-p^{-1})^{\frac{1}{2}}\bar{c}_i\cdot x(log\ x)^{-\frac{1}{2}},\ 1\leq i\leq 4$$
and all calculations are done to the one million-th prime.

When $M=5$, let $m_1=1,\ m_2=2,\ m_3=3,\ m_4=4$. We have the following result:
$$\bar{c}_1=1.14715,\ \bar{c}_2=0.30951,\ \bar{c}_3=0.557665,\ \bar{c}_4=1.28255,$$
\begin{align*}
    &(1,1,1,1)<(1,0,1,1)<(1,1,0,1)<(1,0,0,1)<\\
    &(0,1,1,1)<(1,1,1,0)<(0,0,1,1)<(1,0,1,0)<\\
    &(0,1,0,1)<(1,1,0,0)<(0,0,0,1)<(1,0,0,0)<\\
    &(0,1,1,0)<(0,0,1,0)<(0,1,0,0)<(0,0,0,0).
\end{align*}

When $M=8$, let $m_1=1,\ m_2=3,\ m_3=5,\ m_4=7$. We have the following result:
$$\bar{c}_1=1.43848,\ \bar{c}_2=0.570731,\ \bar{c}_3=0.807147,\ \bar{c}_4=1.01717,$$
\begin{align*}
    &(1,1,1,1)<(1,0,1,1)<(1,1,0,1)<(1,1,1,0)<\\
    &(1,0,0,1)<(0,1,1,1)<(1,0,1,0)<(1,1,0,0)<\\
    &(0,0,1,1)<(0,1,0,1)<(1,0,0,0)<(0,1,1,0)<\\
    &(0,0,0,1)<(0,0,1,0)<(0,1,0,0)<(0,0,0,0).
\end{align*}

When $M=10$, let $m_1=1,\ m_2=3,\ m_3=7,\ m_4=9$. We have the following result:
$$\bar{c}_1=1.23125,\ \bar{c}_2=0.598548,\ \bar{c}_3=0.996604,\ \bar{c}_4=1.37658,$$
\begin{align*}
    &(1,1,1,1)<(1,0,1,1)<(1,1,0,1)<(0,1,1,1)<\\
    &(1,1,1,0)<(1,0,0,1)<(0,0,1,1)<(1,0,1,0)<\\
    &(0,1,0,1)<(1,1,0,0)<(0,1,1,0)<(0,0,0,1)<\\
    &(1,0,0,0)<(0,0,1,0)<(0,1,0,0)<(0,0,0,0).
\end{align*}

We conjecture that there is always a strict bias between $|A_a(x)|$ and $|A_{a'}(x)|$ when $a$ and $a'$ are distinct, that is, there is a strict total ordering on $A=(\mathbb{Z}\slash2\mathbb{Z})^h$. Obviously, $(0,\dots,0)$ is the largest element in this ordering and $(1,\dots,1)$ is the smallest element in this ordering.
%Since we have $L_{A_a}(s)=\frac{1}{L}\sum_{j\in J}\left(\prod_{1\leq i\leq m}e_i^{-a_i j_i}L_j(s)\right)$, each A
\subsection{$l_i\geq 3$ for Some $i$}
Let $l$ be the largest value among ${l_i}$. Then the second highest order pole among $L_j(s)$ would be $\frac{1}{m}((m-1)+e_j)$ and $\frac{1}{m}((m-1)+\bar{e}_j)$ (by highest we mean the value of its real part being largest). 

Hence by \ref{1} we know that $A_j\sim c_j\cdot xRe((log\ x)^{\frac{e_j-1}{m}})$. Just like 4.1, this implies unless the coefficients happen to cancel out, we can expect the sign of $|A_a(x)|-|A_{a'}(x)|$ to change infinitely many times.

We'll consider the following example: 

Let $M=4$, $m_1=1$, $m_2=3$, let $l_1=4$, and $l_2=1$. Let $a=(1,0)$ and $a' =(3,0)$. Then $L_{A_a}(s)-L_{A_a'}(s)=\frac{i}{2}(L_{(3,0)}(s)-L_{(1,0)}(s))=Im(L_{1,0}(s))$.
$$|A_a(x)|-|A_{a'}(x)|\sim Im\left(\frac{1}{\Gamma(\frac{1+i}{2})}\cdot c\cdot x(log\ x)^{\frac{-1+i}{2}}\right)$$
where\begin{align*}
    c&=(\frac{1}{2})^{\frac{1+i}{2}}\lim_{s\rightarrow 1}\prod_{p\equiv 1(4)}((1+ip^{-s})(1-p^{-s})^{\frac{1+i}{2}})\prod_{p\equiv 3(4)}((1+p^{-s})(1-p^{-s})^\frac{1+i}{2})
\end{align*}
As $c\neq 0$, the sign of $|A_a(x)|-|A_{a'}(x)|$ changes infinitely many times as $x\rightarrow \infty$. 

We conjecture that the sign of $|A_a(x)|-|A_{a'}(x)|$ changes infinitely many times for all distinct $a$, $a'$. 
\section*{Acknowledgments}
The author would like to thank Xi Ping for sharing the knowledge on Selberg-Delange method for this problem, and his advisor Wu Han for his guidance and advice on the writing of this paper.

\printbibliography

\end{document}